\documentclass[11pt]{amsart}
\usepackage{amssymb}

\usepackage{palatino}
\input amssym.def

\usepackage{amsmath, amsfonts}
\usepackage{amssymb}
\usepackage{amscd}
\usepackage[mathscr]{eucal}
\usepackage{palatino}
\newfont{\cyrr}{wncyr10}

\newcommand{\N}{{\mathbb N}}
\newcommand{\Z}{{\mathbb Z}}
\newcommand{\Q}{{\mathbb Q}}
\newcommand{\R}{{\mathbb R}}
\newcommand{\C}{{\mathbb C}}

\newcommand{\SL}{{\rm SL}}
\newcommand{\Sp}{{\rm Sp}}
\newcommand{\tab}{\qquad }

\newtheorem{thm}{Theorem}

\newtheorem{lem}[thm]{Lemma}
\newtheorem{cor}[thm]{Corollary}
\newtheorem{prop}[thm]{Proposition}
\newtheorem{rmk}{Remark}[section]

\newcommand{\thmref}[1]{Theorem~\ref{#1}}
\newcommand{\corref}[1]{Corollary~\ref{#1}}
\newcommand{\propref}[1]{Proposition~\ref{#1}}
\newcommand{\lemref}[1]{Lemma~\ref{#1}}
\newcommand{\rmkref}[1]{Remark~\eqref{#1}}

\begin{document}

\title[Hecke eigenvalues]{Arithmetic behaviour of
Hecke eigenvalues of Siegel cusp forms of degree two}

\author{Sanoli Gun, Winfried Kohnen and Biplab Paul}

\address[Sanoli Gun, Biplab Paul]  
{Institute of Mathematical Sciences, 
Homi Bhabha National Institute, 
C.I.T Campus, Taramani, 
Chennai  600 113, India.}

\email{sanoli@imsc.res.in}
\email{biplabpaul@imsc.res.in}

\address[Winfried Kohnen]  
{Mathematisches Institut der Universit\"at, 
INF 288, D-69120, Heidelberg, Germany.}

\email{winfried@mathi.uni-heidelberg.de}



\begin{abstract} 
Let $F$ and $G$ be Siegel cusp forms for $\Sp_4(\Z)$ and 
weights $k_1, k_2$ respectively. Also let
$F$ and $G$ be Hecke eigenforms lying in distinct
eigen spaces. Further suppose that
neither $F$ nor $G$ is a Saito-Kurokawa lift. 
In this article, we study simultaneous
arithmetic behaviour of Hecke eigenvalues of these
Hecke eigenforms.
\end{abstract}

\maketitle

\section{Introduction}

\smallskip

Let $k$ be a positive integer, $\Gamma_2 := \rm Sp_4(\Z)$
be the full Siegel modular group of degree~$2$ and
$S_k(\Gamma_2)$ be the space of Siegel cusp 
forms of weight $k$ for $\Gamma_2$. It is well known (see \cite{DZ}) 
that when $k$ is even, $S_k(\Gamma_2)$ has a canonical subspace 
which is generated by the Saito-Kurokawa lift of Hecke eigenforms in the space of
elliptic cusp forms of weight $2k-2$ for $\SL_2(\Z)$. This subspace is called the 
Maass subspace. When $k$ is odd, we shall define the zero subspace of $S_k(\Gamma_2)$
as Maass subspace. In both the cases, we shall denote these Maass subspaces
by $S_k^*(\Gamma_2)$. 
If $F \in S_k^*(\Gamma_2)$ is a Hecke 
eigenform with eigenvalues $\mu_F(n)$, then one 
knows that $\mu_F(n) >0$ for all $n \in \N$ (see 
\cite{SBr}, also see \cite[Corollary 1.5]{GPS}). 
On the contrary, if $F$ is an Hecke eigenform lying in the 
orthogonal complement of $S_k^*(\Gamma_2)$ in
$S_k(\Gamma_2)$, then the second author \cite{WK} 
showed that the sequence $\{\mu_F(n)\}_{n \in \N}$ changes sign 
infinitely often.

Now suppose that 
$F \in S_{k_1}(\Gamma_2)$ and $G \in S_{k_2}(\Gamma_2)$  
are Hecke eigenforms with eigenvalues 
$\{ \mu_F(n) \}_{n \in \N}$ and $\{ \mu_G(n) \}_{n \in \N}$ 
respectively. In this article, we will investigate
arithmetic properties  of the 
sequence $\{ \mu_F(n) \mu_G(n) \}_{n\in \N}$.
Unlike the elliptic case, it is not known that if
$F$ is not a constant multiple of $G$, then there exists
a natural number $n_0$ such that 
$\mu_F(n_0) \ne \mu_G(n_0)$ (see \cite{{SB},{AS}}). 
Henceforth, we shall assume that 
$F$ and $G$ lie in different eigenspaces.  
We shall also assume that $F \in S_{k_1}(\Gamma_2)$ and $G \in S_{k_2}(\Gamma_2)$
are Hecke eigenforms lying in
the orthogonal complement of the Maass subspace
as the arithmetic properties investigated in this article
are already well understood
for Hecke eigenforms inside the Maass subspace.

We start by investigating the first non-vanishing
of the sequence $\{ \mu_F(p^n) \mu_G(p^n) \}_{n\in \N}$,
when $p$ is a prime. More precisely, we have the 
following theorem.

\begin{thm}\label{th2}
Let $F \in S_{k_1}(\Gamma_2)$ and $G \in S_{k_2}(\Gamma_2)$
be Hecke eigenforms lying in the orthogonal complement of the 
Maass subspace with Hecke eigenvalues $\{ \mu_F(n) \}_{n \in \N}$
and $\{ \mu_G(n) \}_{n \in \N}$ respectively. 
Also let $F$ and $G$ lie in different eigenspaces.
Then for any prime $p$, 
there exists an integer $n$ with $1\le n \le 14$ such that 
$$
\mu_F(p^n)\mu_G(p^n) \ne 0.
$$
\end{thm}

Next, we investigate the growth of the sequence
of normalized Hecke eigenvalues and prove the
following theorem.

\begin{thm}\label{prop1}
Let $F \in S_{k_1}(\Gamma_2)$ and $G \in S_{k_2}(\Gamma_2)$
be Hecke eigenforms lying in the orthogonal complement of the 
Maass subspace and having normalized Hecke eigenvalues 
$\{ \lambda_F(n) \}_{n \in \N}$
and $\{ \lambda_G(n) \}_{n \in \N}$ respectively. 
Also let $F$ and $G$ lie in different eigenspaces.
Then for sufficiently large $x$ and any $\epsilon > 0$, one has 
$$
\sum_{m\le x} \lambda_F(m)\lambda_G(m)
\ll_\epsilon 
\text{ max}\{k_1, k_2\}^{3/8} x^{31/32 + \epsilon},
$$
where the constant in $\ll_\epsilon$ depends only on $\epsilon$.
\end{thm}

As a corollary, we then derive the following.
\begin{cor}\label{th1}
Let $F \in S_{k_1}(\Gamma_2)$ and $G \in S_{k_2}(\Gamma_2)$
be Hecke eigenforms lying in the orthogonal complement of the 
Maass subspace and having Hecke eigenvalues $\{ \mu_F(n) \}_{n \in \N}$
and $\{ \mu_G(n) \}_{n \in \N}$ respectively. 
Also let $F$ and $G$ lie in different eigen spaces.
Then for any $\epsilon >0$, one has
$$
\#\{n \le x ~|~  \mu_F(n) \ne \mu_G(n) \}
\gg
 x^{1 -\epsilon},
$$
where the constant $\gg$ depends on $F, G$ and $\epsilon$.
\end{cor}

Next we investigate the question of Hecke 
eigenvalues which are of different sign. 
Here we have the following theorem;

\begin{thm}\label{th3}
Let $F \in S_{k_1}(\Gamma_2)$ be a Hecke eigenform 
lying in the orthogonal complement of the Maass 
subspace and having Hecke eigenvalues $\{ \mu_F(n) \}_{n \in \N}$.
Also assume that there exist $0 < c < 4$ 
and a Hecke eigenform 
$G \in S_{k_2}(\Gamma_2)$ lying in the 
orthogonal complement of the Maass subspace 
with Hecke eigenvalues $\{ \mu_G(n) \}_{n \in \N}$
such that 
\begin{equation}\label{eq-lb}
\# \{ p \le x  ~|~~  p~\text{prime},~|\mu_G(p)| > c p^{k_2-\frac{3}{2}}\}
~\ge~
\frac{16}{17}\cdot \frac{x}{\log x}
\end{equation}
for sufficiently large $x$. 
Also assume that $F$ and $G$ lie in different eigenspaces.
Then half of the non-zero coefficients 
of the sequence $\{\mu_F(n) \mu_G(n)\}_{n \in \N}$ 
are positive and half of them are negative. 
\end{thm}

We note that the subset  of primes $\{ p  ~|~ \mu_G(p) =0 \}$ 
has density zero (see appendix of \cite{RSW}). Further
the Generalized Ramanujan-Petersson conjecture
proved by Weissauer (\cite{RW}) gives that for any prime
$p$, $|\mu_G(p)| \le 4p^{k_2-\frac{3}{2}}$.  Thus
the hypothesis in \eqref{eq-lb} is not an unreasonable 
one (especially if one also believes
an analogous Sato-Tate conjecture in this setup).  
Now if we restrict to $p$-eigenvalues, then we can 
prove the following theorem;

\begin{thm}\label{th4}
Let $F \in S_{k_1}(\Gamma_2)$ and $G \in S_{k_2}(\Gamma_2)$
be as in \thmref{th3}. Then there exists a set of primes $p$ of positive 
lower density such that $\mu_F(p)\mu_G(p) \gtrless 0$.
\end{thm}

The article is distributed as follows. In the next section, we introduce
notations and preliminaries. In the last few sections, we give
proofs of \thmref{th2}, \thmref{prop1}, \thmref{th3} and \thmref{th4}.

We note that proof of \thmref{th2} requires intricate understanding
of Hecke relations whereas the proof of \thmref{prop1} uses a result
of the first author with R. Murty \cite{GM} and a beautiful work
of Pitale, Saha and Schmidt \cite{PSS} along with some elementary
analytic tools.  Moreover, \thmref{prop1} can be thought of
a generalization of a work of Das, the second author and
Sengupta \cite{DKS}. Proof of \thmref{th4} requires some standard analytic 
techniques and proof of \thmref{th3} is rather 
straightforward from the works of Matom\"aki and Radziwi\l\l ~\cite{MR}
and that we keep it here for the sake of completeness.

\smallskip

\section{Notations and Preliminaries}

\smallskip

Throughout the paper, $\R, \R_+, \Z, \N$ and $\mathcal P$ 
denote the set of real numbers, the set of positive 
real numbers, the set of integers, the set of natural numbers 
and the set of prime numbers respectively. Also we shall use the
symbol $p$ to denote a prime number.

For $f, g : \R \to \C$ with $g(x) > 0$
for all $x \in \R$, we say $f = o(g)$ if $|f(x)|/g(x) \to 0$
as $x \to +\infty$.


We say a subset $A$ of $\mathcal P$ has natural 
density $\alpha \in \R$ if 
$$
\lim_{ x \to \infty} 
\frac{ \#\{ p \in A ~|~ p \le x \}}{\#\{ p \in \mathcal P ~|~ p \le x \}}
$$
exists and is equal to $\alpha$. We shall denote the natural density of 
$A \subset \mathcal P$ by $d(A)$ if it exists. 

We say the density of $A \subset \N$ is $d(A)$ if 
$$
\lim_{x\to \infty} 
\frac{ \#\{n \le x ~|~ n \in A \}}{ \#\{n \le x ~|~ n \in \N \}}
$$
exists and is equal to the real number $d(A)$. 

Throughout the paper, we shall use definitions and basic facts about 
Siegel modular forms. We refer to Andrianov \cite{AA} for further 
details. For any integer $n \in \N$, the Hecke operator $T(n)$ on the space 
$S_k(\Gamma_2)$ is defined by 
$$
T(n)F 
:= 
n^{2k - 3}
\sum_{\gamma \in \Gamma_2 \backslash \mathcal O_{2,n}} 
F ~|~ \gamma,
$$
where 
$$
\mathcal O_{2,n} 
:= 
\{ \gamma \in M_4(\Z) ~|~ \gamma^t J \gamma = nJ \}, 
\phantom{m}
J := 
\left( 
\begin{matrix}
0 & 1_2 \\ 
-1_2 & 0
\end{matrix}
\right).
$$
It is known that the space $S_k(\Gamma_2)$ has a basis of Hecke eigenforms. 
Let $F \in S_k(\Gamma_2)$ be such that $T(n)F = \mu_F(n) F$ for all $n \in \N$. 
Then one knows that $\mu_F$ is a multiplicative function. 
If $F \in S_k(\Gamma_2)$ is not a Saito-Kurokawa lift, by a famous work of 
Weissauer \cite{RW}, one also knows that the generalized Ramanujan-Petersson 
conjecture is true, i.e. for any $\epsilon >0$, one has
$$
\mu_F(n) \ll_{\epsilon} n^{k - 3/2 + \epsilon}.
$$
We shall normalize these eigenvalues and define 
for any $n \in \N$
$$
\lambda_F(n) : = \frac{\mu_F(n)}{n^{k - 3/2}}.
$$
To each Hecke eigenform $F \in S_k(\Gamma_2)$, 
Andrianov \cite{AA} associated 
a $L$-function which is now known as spinor zeta function as follows: 
\begin{equation}
Z_F(s)
:= 
\zeta(2s+1) \sum_{n = 1}^\infty \frac{\mu_F(n)}{n^{s+k-3/2}}.
\end{equation}
The series $Z_F(s)$ is absolutely convergent  and has 
an Euler product in the region $\Re(s)>1$. 
In fact, by the works of Andrianov \cite{AA} and Oda \cite{TO}, 
one knows that if $F$ is not a 
Saito-Kurokawa lift, then the function $Z_F(s)$ is entire 
and that for $n \ge 3$
\begin{equation}
\lambda_F(p^n) 
= 
\lambda_F(p)\lambda_F(p^{n-1})
- \left[ \lambda^2_F(p) - \lambda_F(p^2) - \frac{1}{p} \right] \lambda_F(p^{n-2}) 
+ \lambda_F(p)\lambda_F(p^{n-3}) - \lambda_F(p^{n-4})
\end{equation}
with the assumption that $\lambda_F(p^{n-m}) = 0$ for 
$n < m$. As in the elliptic case, 
by a work of Kowalski and Saha \cite[Appendix]{RSW}, we 
have the following theorem. 

\begin{thm}\rm[Kowalski and Saha]\label{KS-th}
Let $F \in S_{k}(\Gamma_2)$ be a Hecke eigenform with eigenvalues 
$\mu_F(n) $ for $n \in \N$. Also assume that $F$ lies in the orthogonal 
complement of Maass subspace. Then there exists $\delta > 0$ such that 
$$
\# \{ p \le x ~|~ \mu_F(p) = 0 \} 
\ll \frac{x}{ ( \log x)^{1+\delta}}.
$$
\end{thm}

Let $F \in S_{k_1}(\Gamma_2)$ and $G \in S_{k_2}(\Gamma_2)$ be 
Hecke eigenforms lying in the orthogonal complement of Maass subspace. 
Also let $\{ \lambda_F(n) \}_{n \in \N}$ and $\{ \lambda_G(n) \}_{n \in \N}$ 
be the sets of normalized Hecke eigenvalues of $F$ and $G$ 
respectively. Further 
assume that $Z_F(s)$ and $Z_G(s)$ are the spinor zeta functions 
associated to $F$ and $G$ respectively. We then have 
\begin{eqnarray}\label{Spinor}
Z_F(s) 
:= \zeta(2s+1)\sum_{n = 1}^\infty \frac{\lambda_F(n)}{n^s}
:= \prod_{p \in \mathcal P}\prod_{ i =1}^4 \left(1 - \alpha_{p,i}p^{-s}\right)^{-1}\\
\phantom{m}\text{ and }\phantom{m} 
Z_G(s) 
:= \zeta(2s+1)\sum_{n = 1}^\infty \frac{\lambda_G(n)}{n^s } 
:= \prod_{p \in \mathcal P}\prod_{ i =1}^4 \left(1 - \beta_{p,i}p^{-s}\right)^{-1}. \nonumber
\end{eqnarray}
By the work of Weissauer \cite{RW}, we know that 
$|\alpha_{p, i}| = 1 = |\beta_{p,j}|$  for all $1 \le i, j \le 4$. 
Now define the Rankin-Selberg $L$-function $L(F\times G, s)$ 
as follows:
\begin{equation}\label{Ran-Selberg}
L(F \times G, s) 
:= 
\prod_{p \in \mathcal P} \prod_{ 1 \le i, j\le 4} 
\left(1 - \alpha_{p,i}\beta_{p,j}p^{-s}\right)^{-1}.
\end{equation}
This Euler product is absolutely convergent for $\Re(s) > 1$. In fact, 
Pitale, Saha and Schmidt \cite[Theorem C, p. 14]{PSS} 
proved the following theorem
for Hecke eigenforms 
which do not belong to the Maass subspace. 

\begin{thm}\rm [Pitale, Saha and Schmidt]\label{PSS}
Let $F \in S_{k_1}(\Gamma_2), G \in S_{k_2}(\Gamma_2), Z_F(s)$ and $Z_G(s)$ be 
as above. Define the $L$-function $L(F\times G, s)$ as in \eqref{Ran-Selberg}.
Then the infinite product in \eqref{Ran-Selberg} is absolutely 
convergent for $\Re(s)>1$ and the function $L(F \times G, s)$
has meromorphic continuation to $\C$ and is non-vanishing 
on the line $\Re(s) = 1$. Moreover, the function $L(F \times G, s)$ 
is entire except in the case when $k_1 = k_2$ and $\mu_F(n) = \mu_G(n)$ 
for all $n \in \N$. In later case, the function $L(F\times G, s)$ 
has a simple pole at $s = 1$.
\end{thm}

For Hecke eigenforms $F \in S_{k_1}(\Gamma_2)$ and 
$G \in S_{k_2}(\Gamma_2)$ as above with normalized
eigenvalues $\{ \lambda_F(n) \}_{n \in \N}$ 
and $\{ \lambda_G(n) \}_{n \in \N}$ respectively, define 
\begin{equation}\label{nRS}
L(F, G; s) 
:=
\sum_{n = 1}^\infty 
\frac{ \lambda_F(n)\lambda_G(n)}{n^{s}}.
\end{equation}
Note that this series $L(F, G; s)$ is absolutely convergent for $\Re(s) > 1$. In fact, 
the second author along with Das and Sengupta \cite{DKS} proved that the function 
$L(F, F; s)$ has meromorphic continuation to $\Re(s) > 1/2$ with only a simple pole 
at $s =1$. Also they proved the following theorem. 
\begin{thm}\rm [Das, Kohnen and Sengupta]\label{DKS}
Let $F \in S_k(\Gamma_2)$ be a Hecke eigenform which does not lie in the 
Maass subspace with normalized Hecke eigenvalues $\{ \lambda_F(n)\}_{n\in \N}$.
Then for sufficiently large $x$ and any $\epsilon > 0$, we have 
$$
\sum_{n \le x}\lambda_F^2(n) 
= 
c_F x 
+ 
O\left(k^{5/16}x^{31/32 + \epsilon}\right),
$$
where $c_F > 0$ is the residue of the $L$-function $L(F, F; s)$ at $s = 1$. 
\end{thm}

To prove \corref{th1}, we investigate analytic properties 
of $L(F, G; s)$ when $F$ and $G$ lie in different
eigenspaces. In order to do so, we need the 
following result on the formal power series 
by the first author and Ram Murty \cite[Theorem 2]{GM}.

\begin{thm}\rm[Gun and Murty] \label{GM1}
Let $P_i(T)$ and $Q_i(T)$ be non-zero polynomials over $\C$ 
such that degree of $P_i$ is strictly less than the degree of $Q_i$ 
for $i = 1, 2$. Also let 
$$
Q_1(T) 
:= 
\prod_{i = 1}^r (1 - \alpha_i T )^{\ell_i}
\tab \text{ and }\tab
Q_2(T) 
:= 
\prod_{j = 1}^t (1 - \beta_j T )^{m_j},
$$
where $\alpha_i$'s are distinct for $1\le i\le r$ and $\beta_j$'s 
are distinct for $1\le j\le t$ and $\ell_i , m_j \in \N$. Let us also 
assume that 
$$
\sum_{n \ge 0} a_n T^n 
= 
\frac{P_1(T)}{Q_1(T)}
\tab \text{ and }\tab
\sum_{n \ge 0} b_n T^n 
= 
\frac{P_2(T)}{Q_2(T)}
$$
where $a_n, b_n \in \C$ for all $n \ge 0$. Then we have 
$$
\sum_{n \ge 0} a_n b_n T^n 
= 
\frac{R(T)}{\prod_{i,j}(1 - \alpha_i \beta_j T)^{\ell_i m_j}},
$$
where $R(T) \in \C[T]$. Now if $a_0 = 1 = b_0$, then $R(0) = 1$. Further if we 
have $P'_1(0) = 0 = P'_2(0)$, then $R'(0) = 0$. Here $P'$ denotes the derivative 
of $P(T)$ with respect to $T$. 
\end{thm}

To prove \thmref{th3}, we shall make use of the following
result on the sign changes of multiplicative functions by 
Matom\"aki and Radziwi\l\l  ~\cite[Lemma 2.4]{MR}.
\begin{lem}\rm[Matom\"aki and Radziwi\l\l]\label{MR}
Let $K , L : \R_+ \to \R_+$ be functions such that $K(x) \to 0$ and 
$L(x) \to \infty$ as $x \to \infty$. Let $g : \N \to \R$ be a 
multiplicative function such that for every $x \ge 2$, we have  
$$
\sum_{ p \ge x, \atop g(p) = 0} \frac{1}{p} \le K(x)
\phantom{m}\text{ and }\phantom{m}
\sum_{ p \le x, \atop g(p) < 0} \frac{1}{p} \ge L(x).
$$
Then we have 
\begin{eqnarray*}
\#\{ n \le x ~|~ g(n) > 0\} 
& = &
(1 + o(1)) \cdot \#\{ n \le x ~|~ g(n) < 0\} \\
& = &
\left(\frac{1}{2} + o(1)\right) x 
\prod_{p \in \mathcal P} \left(1-\frac{1}{p}\right)
\left(1+ \frac{h(p)}{p} + \frac{h(p^2)}{p^2} + \cdots\right),
\end{eqnarray*}
where $h$ is the characteristic function of the set 
$\{ n \in \N ~|~ g(n) \ne 0 \}$.  
\end{lem}

\smallskip

\section{Proof of \thmref{th2}}

\smallskip

In this section, we shall give a proof of \thmref{th2}. Let us 
recall that for any prime $p$ and any natural 
number $n \ge 3$, we have 
\begin{equation}\label{Hecke}
\lambda_F(p^n) 
= 
\lambda_F(p)\lambda_F(p^{n-1})
- \left[ \lambda^2_F(p) - \lambda_F(p^2) - \frac{1}{p} \right] 
\lambda_F(p^{n-2}) 
+ \lambda_F(p)\lambda_F(p^{n-3}) - \lambda_F(p^{n-4}),
\end{equation}
with the assumption that $\lambda_F(p^{n-m}) = 0$ for $n < m$ are
natural numbers. Similar relations hold among the Hecke
eigenvalues $\lambda_G(p^n)$ for $n\ge 3$. We use these
relations to derive some important consequences which will help us to
prove our result.  We start with a
general result which might be of independent interest.

\begin{lem}\label{poly}
Let $f_0(x) = - 1$ and $f_1(x) = - x$ be polynomials over $\Z$. 
Define a family of polynomials $\{ f_n \}_{n \in \N}$ by
\begin{equation}\label{recurrance}
f_{n+1}(x) = x f_n(x) - f_{n-1}(x).
\end{equation}
Then for any $\alpha \in \Q \setminus \Z$, we have $f_n(\alpha) \ne 0$ 
for all $n \in \N$.
\end{lem}

\begin{proof}
We first show by induction on $n \in \N$ that
\begin{equation}\label{rec}
f_n(x) 
= 
- x^n + a_{n, n - 1}x^{n - 1} + a_{n, n - 2}x^{n - 2} + \cdots 
+  a_{n, 1}x + a_{n, 0},
\end{equation}
where $a_{n, i} \in \Z, ~ 0 \le i \le n - 2$. 
Note that this is true for $n = 0, 1$. 
Using \eqref{recurrance}, we have
$$
f_{n+1}(x) 
~=~
- x^{n+1} + a_{n, n-1} x^{n} + (a_{n, n-2} + 1) x^{n - 1} +
\cdots + ( a_{n, 0} - a_{n-1, 1})x - a_{n-1, 0}.
$$
Hence by induction we have \eqref{rec}.
Since $\Z$ is integrally closed, any 
solution in $\Q$ of $f_n(x) =0$ for any $n$ will be an integer.
This completes the proof of the lemma.
\end{proof}

\begin{lem}\label{even}
Let $F \in S_k(\Gamma_2)$ be a Hecke eigenform which lies in the 
orthogonal complement of the Maass subspace with normalized Hecke 
eigenvalues $\lambda_F(n)$ for $n \in \N$. 
Then 
\begin{enumerate}

\item
If $\lambda_F(p^{2m}) = 0$ for some $m \ge 2$, then
at least one of $\lambda_F(p), \lambda_F(p^2)$ is non-zero.
 
\item
There does not exist $t \in \N$ such that
$$
\lambda_F(p^m) ~=~0 
\phantom{m}\text{for}\phantom{m}
t+1\le m \le t+4.
$$
\end{enumerate}
\end{lem}

\begin{proof}
Suppose that $\lambda_F(p)= 0 = \lambda_F(p^2)$. 
Then for any $n \ge 0$,
$$
\lambda_F(p^{2n + 4}) =  f_n \left(\frac{1}{p} \right),
$$
where $f_n$'s are polynomials in $\Z[x]$ 
satisfying the hypothesis of \lemref{poly}.
Hence by \lemref{poly}, we have
$\lambda_F(p^{2m}) \ne 0$ for all $m \ge 2$, a contradiction
to our hypothesis. This completes the proof 
of the first part of the lemma.

To prove the second part of the lemma, let us assume that
$\lambda_F(p^m) =0$ for $t+1\le m \le t+4$.
Using \eqref{Hecke}, we have
$$
\lambda_F(p^t) = - \lambda_F(p^{t+4}) =0.
$$
Using induction and the identity \eqref{Hecke}, 
we get that $\lambda_F(p^m) =0$ for $1 \le m \le t +4$. 
This implies that $\lambda_F(p)= 0 = \lambda_F(p^2)$,
a contradiction to the first part of the lemma.
\end{proof}

\begin{lem}\label{odd}
Let $F \in S_k(\Gamma_2)$ be a Hecke eigenform which lies in the 
orthogonal complement of the Maass subspace with normalized Hecke 
eigenvalues $\lambda_F(n)$ for $n \in \N$. Then
\begin{enumerate}

\item
For some $m \ge 0$, $\lambda_F(p^{2m+ 1}) \ne 0$
implies that $\lambda_F(p) \ne 0$.

\item
If $\lambda_F(p) \ne 0$, then for any 
$m \in \N$, there exists $0 \le i \le 3$ such that
$\lambda_F(p^{2(m + i) +1}) \ne 0$.
\end{enumerate}
\end{lem}

\begin{proof}
We shall show by induction on $m$ that
$\lambda_F(p)=0$ implies that $\lambda_F(p^{2m+1})=0$
for all $m \ge 0$.  It is clearly true for $m=0,1$. 
Using \eqref{Hecke}, we have 
$$
\lambda_F(p^{2m+1}) 
= 
\left[ \lambda_F(p^2) + \frac{1}{p}\right] \lambda_F(p^{2m-1}) 
 - \lambda_F(p^{2m-3}).
$$
By induction hypothesis, one knows that 
$$
\lambda_F(p^{2m - 1}) =0 =\lambda_F(p^{2m - 3})
$$ 
and hence $\lambda_F(p^{2m+1}) = 0$. This completes the 
proof of the first part.

To prove the second part, assume that there exists
$m_0 \in \N$ such that 
\begin{equation}\label{eq-i}
\lambda_F(p^{2(m_0 + i) +1}) = 0
\end{equation}
for all $0 \le i \le 3$. Using \eqref{Hecke} and \eqref{eq-i} for
$i=2,3$, we have
$$
\lambda_F(p^{2m_0 + 6}) 
= - \lambda_F(p^{2m_0 + 4})
= \lambda_F(p^{2m_0 + 2}) 
$$
as $\lambda_F(p) \ne 0$. Again using \eqref{Hecke} 
and \eqref{eq-i}, we get
$$
\lambda_F(p^{2m_0 + 6}) 
~=~
- [\lambda_F^2(p) - \lambda_F(p^2) - \frac{1}{p}]
\lambda_F(p^{2m_0 +4}) - \lambda_F(p^{2m_0 +2}).
$$
Hence
\begin{eqnarray*}
0 = \lambda_F(p^{2m_0 + 6})  + \lambda_F(p^{2m_0 + 4}) 
&=&
- [\lambda_F^2(p) - \lambda_F(p^2) - \frac{1}{p} -1]
\lambda_F(p^{2m_0 +4}) - \lambda_F(p^{2m_0 +2})\\
&=&
- [\lambda_F^2(p) - \lambda_F(p^2) - \frac{1}{p} -2]
\lambda_F(p^{2m_0 +2}).
\end{eqnarray*}
This implies that 
$$
\lambda_F^2(p) - \lambda_F(p^2) - \frac{1}{p} =2
$$
as $\lambda_F(p^{2m_0 +2}) \ne 0$ by second part of 
\lemref{even}. Replacing 
$$
\lambda_F(p^{2m_0 +4}) 
~=~ 
-2\lambda_F(p^{2m_0 +2})  - \lambda_F(p^{2m_0}) 
$$
in the relation
$$
0 = \lambda_F(p^{2m_0 + 5}) 
~=~ 
\lambda_F(p)[\lambda_F(p^{2m_0 +4})  + \lambda_F(p^{2m_0 +2})],
$$
we get $\lambda_F(p^{2m_0 +2}) +  \lambda_F(p^{2m_0}) =0$
as $\lambda_F(p) \ne 0$.
Then 
$$
0= \lambda_F(p^{2m_0 + 3})
~=~
\lambda_F(p)[\lambda_F(p^{2m_0 +2}) +  \lambda_F(p^{2m_0})]
- \lambda_F(p^{2m_0 -1})
~=~
- \lambda_F(p^{2m_0 -1}).
$$ 
This shows that if $\lambda_F(p) \ne 0$ and
$\lambda_F(p^{2(m_0 + i) +1}) = 0$ for all $0 \le i \le 3$ and
for some $m_0\in \N$, then $\lambda_F(p^{2m_0 -1})=0$.
Arguing similarly and using induction, we can now show that
$\lambda_F(p^{2m + 1})=0$ for all $1 \le m \le m_0 + 3$.
Note that
\begin{eqnarray*}
0 ~=~ 
\lambda_F(p^5) 
&=&  
\lambda_F(p)[\lambda_F(p^4) + \lambda_F(p^2) -1]\\
&=&
\lambda_F(p)[ -\lambda_F(p^2) + \lambda_F^2(p) -2] \\
&=&
\frac{1}{p} \lambda_F(p),
\end{eqnarray*}
a contradiction to our hypothesis. This completes the
proof of the second part of \lemref{odd}.
\end{proof}

\begin{rmk}\label{e-power}
Let $F \in S_k(\Gamma_2)$ be a Hecke eigenform which lies in the 
orthogonal complement of the Maass subspace with normalized Hecke 
eigenvalues $\lambda_F(n)$ for $n \in \N$. If
$\lambda_F(p) \ne 0$, then there does not exist $m \in \N$ such that 
$\lambda_F(p^{2(m + i)}) = 0$ for all $0 \le i \le 3$.
\end{rmk}

\begin{proof}
Suppose that there exists $m_0 \in \N$ such that
$$
\lambda_F(p^{2(m_0 + i)}) = 0, 
\phantom{m}\text{for } \phantom{m}
0 \le i \le 3.
$$
Arguing as in \lemref{odd}, then we have 
$2 + 1/p + \lambda_F(p^2) - \lambda_F^2(p) = 0$ and 
$\lambda_F(p^{2m}) = 0$ for $1 \le m \le m_0 + 3$ 
as $\lambda_F(p) \ne 0$. This
implies that $\lambda_F^2(p) = 2 + 1/p$ and hence
$\lambda_F(p^4) = -1$, a contradiction. 
\end{proof}

We now complete the proof of \thmref{th2}.

\begin{proof}
Without loss of generality, we can assume that 
$\lambda_F(p) \lambda_G(p) = 0$ and 
$\lambda_F(p^2) \lambda_G(p^2) = 0$, 
otherwise we are done. 

First suppose that $\lambda_F(p) = \lambda_G(p) = \lambda_F(p^2) 
= \lambda_G(p^2) = 0$. Then using the identity 
\eqref{Hecke}, we see that $\lambda_F(p^4)\lambda_G(p^4)= 1$. 
Hence we are done. 

Now we assume that $\lambda_F(p) = \lambda_G(p) = \lambda_F(p^2) = 0$ 
but $\lambda_G(p^2) \ne 0$. Then 
$$
\lambda_G(p^6) 
= \left[ \lambda_G(p^2) + \frac{1}{p} \right] \lambda_G(p^{4})  
- \lambda_G(p^{2})
$$
implies that either $\lambda_G(p^4) \ne 0$ or $\lambda_G(p^6) \ne0$.
Now using \lemref{even}, we are done. 

Next assume that $\lambda_F(p) = 0 
= \lambda_F(p^2)$ and $\lambda_G(p) \ne 0$.
Using \lemref{even}, we know that
$\lambda_F(p^{2n}) \ne 0$ for all $ n \ge 2$. 
Since $\lambda_G(p) \ne 0$,  by \rmkref{e-power}, we have 
at least one of 
$$
\lambda_G(p^4), ~\lambda_G(p^6), 
~\lambda_G(p^8), ~\lambda_G(p^{10})
$$ 
is non-zero. Hence we are done in this case. 

Finally, we assume that $\lambda_F(p) = 0, \lambda_F(p^2) \ne 0$ 
and $\lambda_G(p) \ne 0, \lambda_G(p^2) = 0$. 
Since $\lambda_F(p) = 0$ we know by Lemma \ref{odd} that
$\lambda_F(p^{2n-1}) = 0$ for all $n \in \N$. 

We first consider the case when $\lambda_F(p^4) = 0$. 
Then using \eqref{Hecke}, we have $\lambda_F(p^n) \ne 0$ 
for $n = 6, 8, 10, 12$.  
Since $\lambda_G(p) \ne 0$, using \rmkref{e-power} we are done.

Now assume that $\lambda_F(p^4) \ne 0$ and $\lambda_G(p^4) = 0$, 
otherwise we are done. We will show in this case that 
$\lambda_G(p^6) \ne 0$ except when $p=2$.
Since $\lambda_G(p^4) = 0$, we get
\begin{equation}\label{eq-id}
[2 + 1/p - \lambda_G^2(p) ]\lambda_G^2(p) = 1.
\end{equation}
Using \eqref{eq-id} and \eqref{Hecke}, we have
\begin{eqnarray*}
\lambda_G(p^6) 
&=& 
- \lambda_G^2(p) + \lambda_G(p)\lambda_G(p^3)
[1 + \frac{1}{p} - \lambda_G^2(p)]\\
&=&
- \lambda_G^2(p) + \lambda_G^2(p^3)
~~=~~
\frac{1}{p} - \lambda_G^2(p).
\end{eqnarray*}
Again using \eqref{eq-id}, we see that 
$1/p - \lambda_G^2(p)=0$ only when $p =2$.
If $\lambda_F(p^6) \ne 0$, we are done except when
$p = 2$. So without loss of generality, we can assume that
$\lambda_F(p^6)=0$ when $p \ne 2$. Then
$$
1 + \lambda_F(p^4)
~=~ 
[\lambda_F(p^2) +\frac{1}{p}] \lambda_F(p^2), 
\phantom{m}
\lambda_F(p^2)
~=~ 
[\lambda_F(p^2) +\frac{1}{p}] \lambda_F(p^4) 
$$
and hence
$$
\lambda_F(p^8) = - \lambda_F(p^4),
\phantom{m}
\lambda_F(p^{10}) = - \lambda_F(p^2),
\phantom{m}
\lambda_F(p^{12}) = -1,
\phantom{m}
\lambda_F(p^{14}) = -\frac{1}{p}.
$$
We are now done by \rmkref{e-power}.

It only remains to prove the case when $p = 2$ and 
$\lambda^2_G(2) = 1/2$. In this case, 
$$
\lambda_G(2^8) = - 1
\phantom{m}\text{and}\phantom{m}
\lambda_G(2^{10}) = - 1/2.
$$ 
Now note that either $\lambda_F(2^8) \ne 0$ or
$\lambda_F(2^8) = 0$ and
$\lambda_F(2^{10}) = - \lambda_F(2^6) \ne 0$. 
This completes the proof of \thmref{th2}. 
\end{proof}

\smallskip

\section{Proof of \thmref{prop1} and \corref{th1}}

\smallskip

In this section,  we shall complete the proof of \thmref{prop1}
and \corref{th1}.  In order to prove \thmref{prop1}, we first establish 
a relation between the functions 
$L(F, G; s)$ and $L(F \times G, s)$. More precisely, we show the following. 

\begin{lem}\label{lem19}
Let $F \in S_{k_1}(\Gamma_2)$ and $G \in S_{k_2}(\Gamma_2)$ be as in  
\thmref{prop1}. Then for $\Re(s) > 1$, one has 
\begin{equation}
L(F, G; s) = g(s) L(F \times G; s),
\end{equation}
where
\begin{equation}\label{eq3.1}
g(s) := \prod_{p \in \mathcal P} g_p(p^{-s}).
\end{equation}
Here $g_p(X)$'s are polynomials of degree $\le 15$
and the Euler product on the right hand side of \eqref{eq3.1}
is absolutely convergent for $\Re(s) > 1/2$.
Further, there exists an absolute constant $A > 0$ such that 
$$
g(s) 
\ll 
\sigma^A\left( \sigma - \frac{1}{2} \right)^{-A}
$$
holds uniformly for any $\sigma :=\Re(s) > 1/2$.
\end{lem}

\begin{proof}
Consider the $L$-functions
$$
L(F, s)
:=
\sum_{n=1}^\infty \frac{ \lambda_F(n)}{n^s}
\tab \text{ and }\tab
L(G, s) 
:= \sum_{n=1}^\infty \frac{ \lambda_G(n)}{n^s}.
$$
These $L$-functions are absolutely convergent for $\Re(s) > 1$ 
and by \eqref{Spinor}, we have
$$
L(F, s) 
= \frac{Z_F(s)}{\zeta(2s + 1)}
\tab \text{ and }\tab
L(G, s) 
= \frac{Z_G(s)}{\zeta(2s + 1)}.
$$
Here $Z_F(s), Z_G(s)$ are the spinor zeta functions associated to 
$F$ and $G$ respectively. Since $\lambda_F(n)$ and $\lambda_G(n)$ 
are multiplicative, again using \eqref{Spinor}, we can write 
\begin{eqnarray*}
\sum_{n = 0}^\infty  \lambda_F(p^n)T^n 
= 
\frac{1 - \frac{1}{p} T^2}{ \prod_{1 \le i \le 4}(1 - \alpha_{p,i} T)}
\tab\text{ and }\tab
\sum_{n = 0}^\infty \lambda_G(p^n)T^n 
= 
\frac{1 - \frac{1}{p} T^2}{\prod_{ 1 \le i \le 4}(1 - \beta_{p,i} T)}.
\end{eqnarray*}
Now by \thmref{GM1}, one has 
$$
\sum_{n = 0}^\infty \lambda_F(p^n)\lambda_G(p^n)T^n
= 
\frac{g_p(T)}{\prod_{1 \le i, j \le 4}(1 - \alpha_{p,i} \beta_{p,j} T)},
$$
where $g_p(T) \in \C[T]$ is a polynomial of degree at most $15$.
Also $g_p(0) = 1$ and $g'_p(0) = 0$, where $g_p'$ is the
derivative of $g_p$. The fact
$| \alpha_{p, i} | = | \beta_{p, j} | = 1$ for $ 1 \le i, j \le 4$, 
implies that the coefficients of $g_p(T)$ are bounded
by an absolute constant.  Since $g_p(0) = 1$, 
the coefficients of $T$ in the polynomial $g_p(T)$ is zero 
and other coefficients are bounded by an absolute
constant, it is easy to conclude that 
$$
\prod_{p \in \mathcal P} g_p(p^{-s})
$$
is absolutely convergent for $\Re(s) > 1/2$. This shows that 
for $\sigma > 1$, we have
$$
L(F, G; s) = L(F \times G; s) g(s).
$$ 
It remains to show that $g(s)$ has the required bound. Let 
$$
g_p(T) := 1 + a(p^2) T^2 + \cdots + a(p^{15}) T^{15}, 
$$
where $a(p^i) \in \C$ and $a(p^i)$ are bounded by an absolute constant
for all $2 \le i \le 15$ and for all $p$. 
Let $A > 0$ be an integer such that $ | a(p^2) | \le A$ for all $p \in \mathcal P$. 
Thus 
$$
| g_p(p^{-s}) | 
= 
\left| 1 + \sum_{2 \le n \le 15} a(p^n) p^{-ns} \right| 
\le h_p(\sigma),
$$
where 
$$
h_p(s) 
:= 
1 + Ap^{-2s} + | a(p^3) | p^{-3s} + \cdots + | a(p^{15}) | p^{-15s}.
$$
Now note that 
\begin{equation}\label{bound}
\left(1 - p^{-2s} \right)^A h_p(s)
= 
1 + O\left( p^{-3\sigma} \right).
\end{equation}
The left hand side of \eqref{bound} is nothing but the $p$-th Euler 
factor of the Dirichlet series 
$$
\zeta(2s)^{-A} h(s), 
\tab \text{ where }\tab 
h(s) := 
\prod_{p \in \mathcal P} h_p(s).
$$
Hence for all $\sigma > 1/2$, we have 
$$
g(s)
\ll 
\left( \frac{ \sigma }{ \sigma - 1/2} \right)^A.
$$
This completes the proof of \lemref{lem19}.
\end{proof}

As an application of the above lemma, one can derive the 
following analytic properties of the $L$-function $L(F, G; s)$. 

\begin{lem}\label{lem1}
Let $F \in S_{k_1}(\Gamma_2)$ and $G \in S_{k_2}(\Gamma_2)$ 
be as in \thmref{prop1}.  Then the function $L(F, G; s)$ 
admits an analytic continuation to $\Re(s) >1/2$. 
\end{lem}

\begin{proof}
We know from \lemref{lem19} that for $\sigma >1$
$$
L(F, G; s) = g(s) L(F \times G, s).
$$
Now holomorphicity of $g(s)$ to $\Re(s) > 1/2$ along with the
fact that $L(F \times G, s)$ has analytic continuation 
to $\C$ (see \thmref{PSS}) implies that $L(F, G; s)$
can be continued analytically upto $\sigma >1/2$.
\end{proof}
To prove \thmref{prop1}, we also need the following 
convexity bound.

\begin{lem}\label{lem20}
Let $F \in S_{k_1}(\Gamma_2)$ and $G \in S_{k_2}(\Gamma_2)$ be as in  
\thmref{prop1}. Then for any $\epsilon > 0$ and $0 < \delta < 1$, one has 
\begin{equation}
L(F\times G, \delta + it) 
\ll_{\epsilon} 
\text{ max}\{ k_1, k_2 \}^{6(1- \delta + \epsilon)} |3 + it~|^{8(1 - \delta + \epsilon)}.
\end{equation}
\end{lem}

To prove \lemref{lem20} we shall use the
following strong convexity principle due to Rademacher.

\begin{prop}\rm [Rademacher]\label{Rademacher}
Let $g(s)$ be holomorphic and of finite order in $a < \Re(s) < b$, 
and continuous on the closed strip $a \le \Re(s) \le b$. 
Also let
$$
| g(a + it) | \le E  | P + a + it |^\alpha 
\phantom{m} \text{ and }\phantom{m}
| g(b + it) | \le F  | P + b + it |^\beta, 
$$
where $E, F$ are positive constants and $P, \alpha,\beta$ are real 
constants satisfying 
$$
P + a > 0, 
\phantom{m}
\alpha \ge \beta.
$$
Then for $a < \sigma  < b $, we have 
$$
| g(s) | 
\le 
(E | P + s |^\alpha )^{\frac{b - \sigma}{b - a}}
(F | P + s |^\beta )^{\frac{\sigma - a}{b - a}}.
$$
\end{prop}

We now complete the proof of \lemref{lem20}.
\begin{proof}
Without loss of generality, let us assume that $k_1 \ge k_2 > 2$. 
It is known by \cite[sec. 5.1]{PSS} that $F$ (also $G$) can 
be associated to a cuspidal, automorphic 
representation $\pi$ (resp. $\pi'$) of $\rm GSp_4(\mathbb A)$ such that  
$\pi$ (resp. $\pi'$) has trivial central character, the archimedean component 
$\pi_\infty$ (resp. $\pi'_\infty$) is a holomorphic discrete series representation with 
scalar minimal $K$-type $( k_1, k_1 )$ [resp. $( k_2, k_2 )$] and
for each finite place $p$, the local representation $\pi_p$ [resp. $\pi_p'$]
is unramified. Here $\mathbb A$ is the ring of adeles of $\Q$.  
The real Weil group $W_\R$ is given by $\C^\times \sqcup j \C^\times$ 
such that $j^2 = -1$ and $ j z j^{-1} = \overline{z}$ for $z \in \C^\times$.  
Then the real Weil group representations underlying Siegel modular
forms $F$ and $G$ of weights $k_1$ and $k_2$ respectively
are given by (see page 90 of \cite{PSS} and page 2397 of \cite{RS})
$\varphi_{2k_1 - 3} \oplus \varphi_1$ and $\varphi_{2k_2 - 3} \oplus \varphi_1$,
where for $k \in \N$, $\varphi_k$ is defined by 
$$
\varphi_k : 
\C^\times \ni re^{i\theta} 
\mapsto 
\left[
\begin{matrix}
e^{ik\theta} &  \\
 		 & e^{-ik\theta}
\end{matrix}
\right], 
\phantom{m}
j 
\mapsto 
\left[
\begin{matrix}
    &  (-1)^k \\
1  & 
\end{matrix}
\right].
$$  
Then the parameter of $\pi_\infty \times \pi'_\infty$ is 
\begin{eqnarray*}
( \varphi_{2k_1 - 3} \oplus \varphi_1 )  
\otimes 
(\varphi_{2k_2 - 3} \oplus \varphi_1) 
 =  
\begin{cases}
\varphi_{2k_1 + 2k_2 - 6} \oplus \varphi_{2(k_1 - k_2)} \oplus \varphi_{2k_1 - 2} 
\oplus \varphi_{2k_1 - 4} \\
\tab 
\oplus \varphi_{2k_2 - 2} \oplus \varphi_{2k_2 - 4} 
\oplus \varphi_2 \oplus \varphi_+ \oplus \varphi_-  
& \text{ if } k_1 > k_2 \\
\varphi_{4k_1 - 6} \oplus \varphi_+ \oplus \varphi_- \oplus \varphi_{2k_1 - 2} 
\oplus \varphi_{2k_1 - 4} \\
\tab
\oplus \varphi_{2k_1 - 2} \oplus \varphi_{2k_1 - 4} 
\oplus \varphi_2 \oplus \varphi_+ \oplus \varphi_-  
& \text{ if } k_1 = k_2. 
\end{cases}
\end{eqnarray*}
Here $\varphi_+ $ and $\varphi_-$ are given by 
\begin{eqnarray*}
\varphi_+ :  
re^{i\theta} \mapsto 1, 
&
j \mapsto 1;\\
\varphi_- :  
re^{i\theta} \mapsto 1, 
&
j \mapsto  - 1.
\end{eqnarray*}
Now from \cite[Table 2]{RS}, one can easily see that the 
gamma factors of $L(F \times G, s)$ are as follows: 
\begin{eqnarray*}
L_\infty(F\times G, s) 
:= 
\begin{cases}
\Gamma_\C(s + k_1 + k_2 - 3) \Gamma_\C(s + k_1 - k_2) 
\Gamma_\C(s + k_1 - 1) \Gamma_\C(s + k_1 - 2) \\
\tab 
 \Gamma_\C(s + k_2 - 1) \Gamma_\C(s + k_2 - 2) 
\Gamma_\C(s + 1) \Gamma_\R(s) \Gamma_\R(s + 1) 
& \text{ if } k_1 > k_2, \\
&\\
\Gamma_\C(s + 2k_1 - 3) \Gamma_\C^2(s + k_1 - 1) 
\Gamma_\C^2(s + k_1 - 2) \\
\tab 
\Gamma_\C(s + 1) \Gamma_\R^2(s) \Gamma_\R^2(s + 1) 
& \text{ if } k_1 = k_2,
\end{cases}
\end{eqnarray*}
where $\Gamma_\R(s) := \pi^{-s/2} \Gamma(s/2)$ and 
$\Gamma_\C(s) := 2 ( 2\pi )^{-s} \Gamma(s)$.
Again by \cite[Theorem 5.2.3]{PSS}, we know that the 
completed $L$-function 
$$
L^*(F \times G, s)
:= 
L_\infty(F\times G, s) L(F\times G, s)
$$
satisfies the functional equation 
$$
L^*(F\times G, 1 - s) 
= 
\epsilon(F \times G, s) L^*(F\times G, s),
$$
where $ \epsilon(F \times G, s) \in \C$ and has absolute value $1$.
Thus for any $s\in \C$ with $\sigma > 1$, we have 
$$
\left| L( F \times G, 1 - s) \right|
= 
\left| 
\frac{L_\infty(F \times G, s)}{L_\infty(F \times G, 1- s)}\right| 
\cdot \left| L( F \times G, s) \right|.
$$
Note that for $s = c + it$ with $1 < c  < 3/2$, we have 
$$
\left| \frac{L_\infty(F \times G, c + it)}{L_\infty(F \times G, 1- c - it)} \right| 
\ll 
k_1^{6(2c - 1)} | 1 + it |^{8(2c - 1)}.
$$
Let $c = 1+ \epsilon$ with $0 < \epsilon < 1/2$.
Since $| L(F \times G, 1 + \epsilon + it ) | \ll_\epsilon 1$,
for any $0 < \delta < 1$, using \propref{Rademacher}, we have 
$$
| L(F \times G, \delta + it ) | 
\ll
k_1^{6(1 - \delta + \epsilon)} 
| 3 + it |^{8(1 - \delta + \epsilon)}.  
$$
This completes the proof of the lemma.
\end{proof}

Now we are ready to prove \thmref{prop1}.

\bigskip
\noindent
{\bf Proof of \thmref{prop1}.}
From the work of Weissauer \cite{RW} one knows that the generalized 
Ramanujan-Petersson conjecture is true for $F$ and $G$ and so for any 
$\epsilon > 0$, one has 
$$
\lambda_F(n)\lambda_G(n)\ll n^{\epsilon}. 
$$
Hence by the Perron's summation formula, we have 
$$
\sum_{n \le x} \lambda_F(n)\lambda_G(n)
= 
\frac{1}{2\pi i} \int_{1 + \epsilon - iT}^{1 + \epsilon + iT}
L(F, G; s) \frac{x^s}{s}ds
+ 
O\left( \frac{x^{1 + 2\epsilon}}{T} \right).
$$
Now we shift the line of integration to $1/2< \Re(s):= \delta <1$ 
(to be chosen later). Since there are no singularities of the 
function $L(F,G; s)x^s / s$ 
in the region bounded by the lines joining the points 
$1+\epsilon-iT, 1+\epsilon+iT, \delta+iT$ and $\delta-iT$, 
we have 
$$
\sum_{n \le x} \lambda_F(n)\lambda_G(n)
= 
I_1 + I_2 + I_3
+ 
O\left( \frac{x^{1 + 2\epsilon}}{T} \right),
$$
where 
\begin{eqnarray*}
I_1 
& := & 
\frac{1}{2\pi i}\int_{\delta-iT}^{{\delta+iT}}
L(F,G; s)\frac{x^s}{s}ds, \tab 
I_2 
:= 
\frac{1}{2\pi i}\int_{\delta+iT}^{{1+\epsilon+iT}}
L(F,G; s)\frac{x^s}{s}ds \\
&& 
\tab \text{ and }\tab
I_3 
:= 
\frac{1}{2\pi i}\int_{1+\epsilon-iT}^{{\delta-iT}}
L(F,G; s)\frac{x^s}{s}ds.
\end{eqnarray*}
Using \lemref{lem19} and \lemref{lem20}, one can 
easily get 
$$
I_1 
\ll_{\epsilon} 
( \delta - 1/2)^{-A} k^{6(1 - \delta + \epsilon)}
x^\delta T^{8(1 - \delta + \epsilon)},
$$
where $k = \text{ max}(k_1, k_2)$.
Similarly, one can get 
$$
I_2, I_3 
\ll_{\epsilon} 
(\delta - 1/2)^{-A} k^{6(1 - \delta + \epsilon)}
x^{1 + \epsilon} T^{8(1 - \delta + \epsilon) - 1}.
$$
We shall put $T=x^\alpha$, where $\alpha>0$ is a real number 
to be chosen later. Thus we have 
$$
\sum_{n \le x} \lambda_F(n) \lambda_G(n) 
\ll_\epsilon
(\delta - 1/2)^{-A} k^{6(1 - \delta + \epsilon)}
\left(
x^{8\alpha( 1 - \delta + \epsilon) + \delta} 
+ 
x^{1 + 8\alpha( 1 - \delta + \epsilon) - \alpha + \epsilon} 
+ 
x^{1 - \alpha + \epsilon}
\right).
$$
Choosing $\alpha = 1/16$ and $\delta = 15/16$, one has 
$$
\sum_{n \le x} \lambda_F(n)\lambda_G(n)
\ll_\epsilon 
k^{3/8 + \epsilon} x^{31/32 + \epsilon}.
$$
This completes the proof of \thmref{prop1}. \qed

\bigskip
\noindent
{\bf Proof of \corref{th1}.}
We know from \thmref{DKS} and \thmref{prop1} that
$$ 
\sum_{n\le x} \lambda^2_F(n) 
~=~
c_Fx  + O(x^{\frac{31}{32}})
\phantom{m}
\text{and}
\phantom{m}
\sum_{n \le x} \lambda_F(n)\lambda_G(n) 
~=~ O(x^{\frac{31}{32}}),
$$
where $c_F>0$. Suppose that $k_2 \le k_1$.
Using partial summation, we get
\begin{equation}\label{eq-sum}
\sum_{n\le x} \mu^2_F(n) 
~=~
cx^{2k_1-2}  + O(x^{2k_1 -2 - \frac{1}{32}})
\phantom{m}
\text{and}
\phantom{m}
\sum_{n \le x} \mu_F(n)\mu_G(n) 
~=~ O(x^{k_1 + k_2-2 - \frac{1}{32}}),
\end{equation}
where $c= \frac{c_F}{2k_1-2}$. Now
let
$$
S(x) 
:=
\sum_{n \le x} [ \mu_F(n) - \mu_G(n) ] \mu_F(n).
$$
Note that for any $\epsilon > 0 $, we have
$$
S(x) 
\le 
c(\epsilon)\cdot
\# \{ n \le x ~|~ n \in \N, ~\mu_F(n) \ne \mu_G(n) \}
x^{2k_1 - 3 + \epsilon} ,
$$
where $c(\epsilon) > 0$ is a constant depending only on $\epsilon > 0$. 
Now by applying \eqref{eq-sum}, we conclude that 
$$
\# \{ n \le x ~|~ n \in \N,~ \mu_F(n) \ne \mu_G(n) \} 
\gg_{F, G, \epsilon} 
x^{1 - \epsilon}.
$$
When $k_1 \le k_2$, we consider the sum
$\sum_{ n \le x} [ \mu_G(n) - \mu_F(n) ] \mu_G(n)$
and proceed as above to get the result.
This completes the proof of \corref{th1}. 

\begin{rmk}
To prove \corref{th1}, we have only used the property 
$$
\sum_{n \le x} \lambda_F(n)\lambda_G(n) 
= o(x),
$$
as $x \to \infty$ but \thmref{prop1} gives an 
explicit upper bound 
and hence is also of independent interest. 
We also note that the \corref{th1} is weaker than the optimal one.
In fact, using identities \eqref{square-prime}, \eqref{prime-no-th}
and the Weissauer bound and proceeding along the same line
of the proof of \corref{th1}, we get
$$
\#\{ p \le x ~|~  p~ \text{prime},~\mu_F(p) \ne \mu_G(p) \}
\ge
\frac{1}{32}\cdot \frac{x}{\log x}.
$$
However our proof follows without appealing to prime number theorem. 
\end{rmk}

\smallskip

\section{Proofs of \thmref{th3} and \thmref{th4}}

\smallskip

In this section, we complete the proofs of \thmref{th3} and \thmref{th4}. 
Let us start with the following lemma. 

\begin{lem}\label{lower-density}
Let $F \in S_{k_1}(\Gamma_2)$ and $G \in S_{k_2}(\Gamma_2)$ 
be Hecke eigenforms in the orthogonal complement of the 
Maass subspace and having normalized eigenvalues 
$\{ \lambda_F(n)\}_{n \in\N}$ 
and $\{ \lambda_G(n) \}_{n\in \N}$ respectively. 
Also assume that $F$ and $G$ lie in different eigenspaces
and there exists $0 < c < 4$
such that
$$
\# \{ p \le x ~|~ |\lambda_G(p)| > c \}
~\ge~ 
\frac{16}{17}\cdot \frac{x}{\log x}
$$
for sufficiently large $x$. Then we have 
$$
\sum_{p \le x} \lambda_F^2(p)\lambda_G^2(p) 
\gg \frac{x}{\log x}.
$$
\end{lem}
\begin{proof}
Note that by \cite[Theorem 5.1.2]{PSS}, one knows that 
the transfers of 
$F$ and $G$ are irreducible unitary cuspidal 
and self-contragredient 
automorphic representations of $\rm GL_4(\mathbb{A})$. 
Hence by \cite[Theorem 3]{WY}, we have 
\begin{equation}\label{square-prime}
\sum_{p \le x}\lambda_F^2(p) 
= \frac{x}{\log x} + o \left( \frac{x}{\log x} \right)
\phantom{m}\text{ and }\phantom{m}
\sum_{p \le x}\lambda_G^2(p) 
= \frac{x}{\log x} + o \left( \frac{x}{\log x} \right),
\end{equation}
as $x \to \infty$. Let $S$ be the set of primes $p$ such that 
$|\lambda_G(p)| > c$. 
Thus for sufficiently large $x$, we have 
$$
\sum_{p \le x}\lambda_F^2(p)\lambda_G^2(p) 
~>~ 
c^2 \sum_{p \le x, \atop p \in S} \lambda_F^2(p).
$$
By the given hypothesis, the set 
$$
\sum_{p \le x, \atop {p \notin S} } 
\lambda_F^2(p) 
\le 
16 \cdot \#\{ p \le x ~|~ p \notin S\}
\le
\frac{16}{17}\cdot \frac{x}{\log x}
$$ 
for sufficiently large $x$. This implies that
$$
\sum_{p \le x}\lambda_F^2(p)\lambda_G^2(p) 
\gg \frac{x}{\log x}
$$
for sufficiently large $x$.
This completes the proof of \lemref{lower-density}.
\end{proof}

We now complete the proof of \thmref{th4} and then use \thmref{th4} 
to complete the proof of \thmref{th3}.

\subsection{Proof of \thmref{th4} }
Using \cite[Theorem 5.1.2]{PSS}, we know that the transfers of 
$F$ and $G$ are irreducible unitary cuspidal and self-contragredient 
automorphic representations of $\rm GL_4(\mathbb{A})$. 
Hence by \cite[Theorem 3]{WY}, we have 
\begin{equation}\label{prime-no-th}
\sum_{p \le x} \lambda_F(p)\lambda_G(p) 
= o \left( \frac{x}{\log x} \right),
\end{equation}
as $x \to \infty$. Consider the sum
$$
S^+(x) 
:= 
\sum_{p \le x} 
[\lambda_F(p)\lambda_G(p) + 16] \lambda_F(p)\lambda_G(p).
$$
Observe that 
\begin{equation}\label{upper}
S^+(x) 
\le 
\sum_{p \le x, \atop \lambda_F(p)\lambda_G(p) > 0} 
[ \lambda_F(p)\lambda_G(p) + 16] \lambda_F(p)\lambda_G(p)
~\le~ 512 \cdot
\#\{ p \le x ~|~  \lambda_F(p)\lambda_G(p) > 0 \}.
\end{equation}
On the other hand, using \lemref{lower-density} and \eqref{prime-no-th}, 
we have for sufficiently large $x$
\begin{equation}\label{den}
S^+(x) 
~=~ 
\sum_{p \le x} \lambda_F^2(p)\lambda_G^2(p) 
+ 
16 \sum_{p \le x} \lambda_F(p)\lambda_G(p) 
~\gg~ 
\frac{x}{\log x}.
\end{equation}
Thus by \eqref{upper} and \eqref{den}, we conclude that there exists 
a set of primes having positive density such that $\lambda_F(p)\lambda_G(p) > 0$. 
Similarly, by considering the sum 
$$
S^-(x)
:=
\sum_{p \le x} 
[ \lambda_F(p)\lambda_G(p) - 16] \lambda_F(p)\lambda_G(p)
$$
and arguing as above one can conclude that there exists a set of primes 
having positive density such that $\lambda_F(p)\lambda_G(p) < 0$. \qed

\subsection{Proof of \thmref{th3}}

It follows from \thmref{KS-th} that there exists $\delta>0$
such that
\begin{eqnarray*}
\#\{ p \le x ~|~ \lambda_F(p)\lambda_G(p) = 0\} 
&\le& 
\#\{ p \le x ~|~ \lambda_F(p) = 0\} 
~+~
\#\{ p \le x ~|~ \lambda_G(p) = 0\} \\
&=& 
O\left(\frac{x}{(\log x)^{1+\delta}}\right)
\end{eqnarray*}
for sufficiently large $x$. Also by \thmref{th4}, we know that
the set $\{ p \in \mathcal P ~|~ \lambda_F(p)\lambda_G(p) < 0 \}$ 
has positive lower density. Hence the multiplicative function 
$\lambda_F(n)\lambda_G(n)$ 
satisfies the hypothesis of \lemref{MR}. We now apply \lemref{MR} to 
complete the proof of \thmref{th3}. \qed

\begin{rmk}
Let $F, G$ be elliptic non-CM cusp forms 
of weights $k_1, k_2$ and levels $N_1, N_2$ respectively. 
Also let $F$ and $G$ be distinct Hecke eigenforms
with eigenvalues $\{ \mu_F(n) \}_{n \in \N}$ and 
$\{ \mu_G(n) \}_{n \in \N}$ respectively.
Then the method adopted here for \thmref{th3} 
can be applied to prove unconditionally that
half of the non-zero coefficients 
of the sequence $\{\mu_F(n) \mu_G(n)\}_{n \in \N}$ 
are positive and half of them are negative. One
can also show unconditionally that there exists a set of 
primes $p$ of positive 
lower density such that $\mu_F(p)\mu_G(p) \gtrless 0$.
\end{rmk}

\bigskip
\noindent
{\bf Acknowledgment:} We would like to thank Ralf Schmidt for 
sending us his paper and useful comments.

\end{document}